\documentclass[a4paper,12pt]{article}

\usepackage[english]{babel}
\usepackage[T1]{fontenc}

\usepackage{tikz}
\tikzstyle{vertex}=[circle, draw, inner sep=2pt, minimum size=6pt]

\usepackage{mathrsfs}

\usepackage[a4paper,top=3cm,bottom=2cm,left=3cm,right=3cm,marginparwidth=1.75cm]{geometry}

\usepackage{comment}
\usepackage{amsfonts}
\usepackage{fullpage}
\usepackage{latexsym}
\usepackage{amsfonts}
\usepackage{blkarray}
\usepackage{graphicx}
\usepackage{float}
\usepackage{enumerate}
\usepackage{amsthm,amsmath,amssymb}
\usepackage{verbatim}
\usepackage{enumitem}
\usepackage{blkarray}
\usepackage{mathtools}
\usepackage{algorithm}
\usepackage[noend]{algpseudocode}

\providecommand{\keywords}[1]{
  \small	
  \textbf{\textit{Keywords---}} #1
}

\newtheorem{lemma}{Lemma}[section]
\newtheorem{corollary}{Corollary}[section]
\newtheorem{proposition}{Proposition}[section]

\newtheorem{assumption}{Assumption}[section]

\newtheorem{example}{Example}[section]

\newcommand{\nc}{\newcommand}

\nc{\cA}{{\cal A}}
\nc{\cB}{{\cal B}}
\nc{\cC}{{\cal C}}
\nc{\cD}{{\cal D}}
\nc{\cE}{{\cal E}}
\nc{\cG}{{\cal G}}
\nc{\cF}{{\cal F}}
\nc{\cH}{{\cal H}}
\nc{\cI}{{\cal I}}
\nc{\cK}{{\cal K}}
\nc{\cL}{{\cal L}}
\nc{\cM}{{\cal M}}
\nc{\cN}{{\cal N}}
\nc{\cO}{{\cal O}}
\nc{\cP}{{\cal P}}
\nc{\cQ}{{\cal Q}}
\nc{\cR}{{\cal R}}
\nc{\cS}{{\cal S}}
\nc{\cT}{{\cal T}}
\nc{\tx}{{\tilde x}}
\nc{\la}{{\langle}}
\nc{\ra}{{\rangle}}
\nc{\ts}{\textsuperscript}
\def\R{\mathbb{R}}

\nc{\bea}{\begin{eqnarray}}
\nc{\eea}{\end{eqnarray}}
\nc{\bean}{\begin{eqnarray*}}
\nc{\eean}{\end{eqnarray*}}
\nc{\be}{\begin{equation}}
\nc{\ee}{\end{equation}}
\nc{\ben}{\begin{equation*}}
\nc{\een}{\end{equation*}}
\nc{\ba}{\begin{array}}
\nc{\ea}{\end{array}}
\title{An Alternative Perspective on Copositive and Convex Relaxations of Nonconvex Quadratic Programs}
\author{E. Alper Y{\i}ld{\i}r{\i}m\thanks{School of Mathematics, Peter Guthrie Tait Road, The University of Edinburgh, Edinburgh, EH9 3FD, United Kingdom. ORCID ID: 0000-0003-4141-3189 E-mail: \tt{E.A.Yildirim@ed.ac.uk}}}
\date{29 May 2020}
\begin{document}

\maketitle

\begin{abstract}
We study convex relaxations of nonconvex quadratic programs. We identify a family of so-called feasibility preserving convex relaxations, which includes the well-known copositive and doubly nonnegative relaxations, with the property that the convex relaxation is feasible if and only if the nonconvex quadratic program is feasible. We observe that each convex relaxation in this family implicitly induces a convex underestimator of the objective function on the feasible region of the quadratic program. This alternative perspective on convex relaxations enables us to establish several useful properties of the corresponding convex underestimators. In particular, if the recession cone of the feasible region of the quadratic program does not contain any directions of negative curvature, we show that the convex underestimator arising from the copositive relaxation is precisely the convex envelope of the objective function of the quadratic program, providing another proof of Burer's well-known result on the exactness of the copositive relaxation. We also present an algorithmic recipe for constructing instances of quadratic programs with a finite optimal value but an unbounded doubly nonnegative relaxation. 
\end{abstract}



\section{Introduction} \label{intro}

In this paper, we are interested in nonconvex quadratic programs that can be represented as follows:
\[
\begin{array}{lrrcl}
(QP) & \ell^* = \min & q(x) &  & \\
 & \textrm{s.t.} & Ax & = & b, \\
 & & x & \geq & 0,
\end{array}
\]
where $q: \R^n \to \R$ is given by
\begin{equation} \label{def_qx}
q(x) = x^T Q x + 2 c^T x.
\end{equation}
The parameters are given by $Q \in \R^{n \times n}$, which is a symmetric matrix, $c \in \R^n$, $A \in \R^{m \times n}$, and $b \in \R^m$, and $x \in \R^n$ is the decision variable. We denote the optimal value of (QP) by $\ell^* \in \R \cup \{-\infty\} \cup \{+\infty\}$, with the usual conventions that $\ell^* = -\infty$ if (QP) is unbounded below, and $\ell^* = +\infty$ if (QP) is infeasible. We denote the feasible region of (QP) by $S$, i.e.,
\begin{equation} \label{def_S}
S = \left\{x \in \R^n: Ax = b, \quad x \geq 0\right\}. 
\end{equation}

Apart from being interesting in their own right, nonconvex quadratic programs also arise as subproblems in sequential quadratic programming algorithms and augmented Lagrangian methods for general nonlinear programming problems (see, e.g.,~\cite{NoceWrig06}).

In this paper, we identify a family of convex relaxations of (QP) that includes the well-known copositive and doubly nonnegative relaxations, with the property that each convex relaxation in this family is feasible if and only if (QP) is feasible. We therefore refer to this family of convex relaxations as {\em feasibility preserving} relaxations. We present an alternative perspective of feasibility preserving convex relaxations of (QP), based on the observation that each such relaxation implicitly gives rise to a convex underestimator of the objective function $q(\cdot)$ on $S$. Our contributions are as follows:

\begin{enumerate}
    \item For each feasibility preserving convex relaxation of (QP), we identify a simple sufficient condition that ensures the exactness of the relaxation and another sufficient condition that results in a trivial convex underestimator.

    \item Under the assumption that $S$ is nonempty and bounded, we show that the convex underestimator arising from the copositive relaxation is in fact the convex envelope of $q(\cdot)$ on $S$. In addition, we show that this conclusion holds even if $S$ is unbounded, under the additional assumption that the recession cone of $S$ does not contain any direction of negative curvature.
    
    \item Under the assumption that $S$ is nonempty and bounded, we show that the convex underestimator arising from the doubly nonnegative relaxation agrees with the objective function value at any local minimizer of (QP).
    
    \item For any $n \geq 5$, we present an algorithmic recipe to construct an instance of (QP) such that $\ell^*$ is finite but the lower bound arising from the doubly nonnegative relaxation equals $-\infty$.
\end{enumerate}

We next briefly review the related literature, with a focus on more recent studies. Semidefinite relaxations of (QP) arising from the "lifting" idea proposed in~\cite{shor87} have been studied extensively in the literature since the early 1990s. More recently, Burer~\cite{Burer09} established that a class of nonconvex quadratic programs, which includes (QP) as a special case, admits an exact copositive relaxation, which unified, extended, and subsumed a number of similar results previously established for more specific classes of optimization problems. Burer's result yields an explicit convex conic representation of a nonconvex optimization problem, where the difficulty is now shifted to the copositive cone that does not admit a polynomial-time membership oracle (see, e.g.,~\cite{ref:MurtKab}). Nevertheless, this unifying result led to intensive research activity along two main directions. On the one hand, Burer's result has been extended to larger classes of nonconvex optimization problems by introducing generalized notions of the copositive cone (see, e.g.,~\cite{burer2012copositive,BurerD12,EichfelderP13,DickinsonEP13,ArimaKK13,PenaVZ15,BaiMP16,ArimaKK16,BomzeCDL17}, and also~\cite{kim2020} for a recent geometric view of copositive reformulations). On the other hand, various tractable inner and outer approximations of the copositive cone have been proposed (see, e.g., \cite{parrilo2000structured,de2002approximation,pena2007computing,bundfuss2009adaptive,alper2012accuracy,lasserre2014new,gouveia2019inner}). In particular, by duality, Burer's result implies that any tractable inner approximation of the copositive cone immediately gives rise to a convex relaxation of (QP). 

Our perspective in this paper allows us to establish several properties of a class of convex relaxations in a unified manner and allows us to pinpoint the relation between a particular property and the relevant underlying structure of the relaxation. We identify several new properties and extend some of the earlier results in the literature. As a byproduct, we obtain an alternative proof of Burer's well-known copositive reformulation result for (QP). Finally, we identify a key property that enables us to construct an instance of (QP) with a finite optimal value and an unbounded doubly nonnegative relaxation.

This paper is organized as follows. We define our notation in Section~\ref{notation}. We review several results about nonconvex quadratic programs and introduce various convex cones in Section~\ref{prelim}. Feasibility preserving convex relaxations and the corresponding convex underestimators are introduced in Section~\ref{feas_pres}. Section~\ref{exact_trivial} presents simple conditions that lead to exact and trivial relaxations. In Section~\ref{bou_vs_unbou}, we study properties of feasibility preserving convex relaxations for quadratic programs with bounded and unbounded feasible regions. Finally, we conclude the paper in Section~\ref{conc}.

\subsection{Notation} \label{notation}

We use $\R^n, \R^n_+$, $\R^{m \times n}$, and $\cS^n$ to denote the $n$-dimensional Euclidean space, the nonnegative orthant, the set of $m \times n$ real matrices, and the space of $n \times n$ real symmetric matrices, respectively. The vector of all ones and the identity matrix are denoted by $e$ and $I$, respectively, whose dimensions will always be clear from the context. We use 0 to denote the real number 0, the vector of all zeroes, as well as the matrix of all zeroes. The convex hull of a set is denoted by $\textrm{conv}(\cdot)$. We reserve uppercase calligraphic letters to denote the subsets of $\cS^n$. For an index set $\mathbf{A} \subseteq \{1,\ldots,n\}$, we denote by $|\mathbf{A}|$ the cardinality of $\mathbf{A}$. For $x \in \R^n$, $Q \in \cS^n$, $\mathbf{A} \subseteq \{1,\ldots,n\}$, and $\mathbf{B} \subseteq \{1,\ldots,n\}$, we denote by $x_\mathbf{A} \in \R^{|\mathbf{A}|}$ the subvector of $x$ restricted to the indices in $\mathbf{A}$ and by $Q_{\mathbf{A}\mathbf{B}}$ the submatrix of $Q$ whose rows and columns are indexed by $\mathbf{A}$ and $\mathbf{B}$, respectively. Therefore, $Q_{\mathbf{A}\mathbf{A}}$ denotes a principal submatrix of $Q$. We simply use $x_j$ and $Q_{ij}$ for singleton index sets. We also adopt Matlab-like notation. We use $1:n$ to denote the index set $\{1,\ldots,n\}$. For $x \in \R^n$ and $y \in \R^m$, we denote by $[x;y] \in \R^{n+m}$ the column vector obtained by stacking $x$ and $y$. For matrices in $\cS^{n+1}$, rows and columns are indexed using $\{0,1,\ldots,n\}$. We use superscripts to denote different elements in a set of vectors or matrices. For any $U \in \R^{m \times n}$ and $V \in \R^{m \times n}$, the trace inner product is denoted by
\[
\langle U, V \rangle := \sum\limits_{i=1}^m \sum\limits_{j = 1}^n U_{ij} V_{ij}.
\]
For any $x \in \R^n$ and $Q \in \cS^n$, note that $x^T Q x = \langle Q, x x^T \rangle$. For $\tilde{x} \in \R^n_+$, we define the following index sets:
\begin{eqnarray} 
\mathbf{P}(\tilde{x}) & = & \left\{j \in \{1,\ldots,n\}: \tilde{x}_j > 0\right\}, \label{def_Px} \\
\mathbf{Z}(\tilde{x}) & = & \left\{j \in \{1,\ldots,n\}: \tilde{x}_j = 0\right\}. \label{def_Zx}
\end{eqnarray}

\section{Preliminaries} \label{prelim}

In this section, we review several results that will be useful for the subsequent exposition.

Consider an instance of (QP). In addition to $S$ given by \eqref{def_S}, which denotes the feasible region of (QP), we define the following sets:
\begin{eqnarray}
L & = & \{d \in \R^n: Ad = 0, \quad d \geq 0\} \label{def_L}\\
L_0 & = & \{d \in \R^n: Ad = 0, \quad d \geq 0, \quad d^T Q d = 0\} \label{def_L0}\\
L_\infty & = & \{d \in \R^n: Ad = 0, \quad d \geq 0, \quad d^T Q d < 0\} \label{def_Linf}
\end{eqnarray}

Note that $L$ denotes the recession cone of $S$. $L_0$ and $L_\infty$ are subsets of $L$ that consist of recession directions of zero and negative curvature, respectively.

If $S = \emptyset$, then (QP) is infeasible and we define $\ell^* = +\infty$. Otherwise, if (QP) is unbounded below, we define $\ell^* = -\infty$. The following well-known result, which completely characterizes unbounded instances of (QP), reveals that (QP) is unbounded below if and only if $q(\cdot)$ is unbounded below on a ray of $S$.

\begin{lemma}\cite[Theorem 3]{Eaves71} \label{unb_qp}
(QP) is unbounded below if and only if $S \neq \emptyset$ and at least one of the following two conditions holds:
\begin{enumerate}
    \item $L_\infty \neq \emptyset$.
    \item There exist ${\tilde d} \in L_0$ and ${\tilde x} \in S$ such that $(Q {\tilde x} + c)^T {\tilde d} < 0$.
\end{enumerate}
\end{lemma}

If $S$ is nonempty and bounded, then $-\infty < \ell^* < +\infty$ and there exists $x^* \in S$ such that $q(x^*) = \ell^*$. If $S$ is an unbounded polyhedron and $q(\cdot)$ is bounded below on $S$, then the well-known result of Frank and Wolfe~\cite{FW56} implies that the optimal value is attained, i.e., there exists $x^* \in S$ such that $q(x^*) = \ell^*$. Therefore, in each of these two cases, we denote the set of optimal solutions of (QP) by
\begin{equation} \label{def_S_star}
S^* = \left\{x^* \in S: q(x^*) = \ell^* \right\}.    
\end{equation}

We next recall two definitions (see, e.g.,~\cite{Roc70}). A function $g: \R^n \to \R$ is said to be a {\em convex underestimator} of $q(\cdot)$ on $S$ if 
\begin{equation} \label{def_conv_undest}
\textrm{$g(\cdot)$ is convex and} \quad g(x) \leq q(x), \quad x \in S.
\end{equation}
The pointwise supremum of all convex underestimators of $q(\cdot)$ on $S$, denoted by $f_S(x)$, is called the convex (lower) envelope of $q(\cdot)$ on $S$:
\begin{equation} \label{def_conv_env}
f_S(x) = \sup \left\{g(x): g(x)~\textrm{is a convex underestimator of $q(\cdot)$ on $S$} \right\}, \quad x \in S.     
\end{equation}
Note that
\begin{equation} \label{conv_env_prop}
\ell^* = \min\limits_{x \in S} q(x) = \min\limits_{x \in S} f_S(x), \quad \textrm{and} \quad \textrm{conv}(S^*) \subseteq \left\{x \in S: f_S(x) = \ell^*\right\}.
\end{equation}
The first relation in \eqref{conv_env_prop} simply follows by combining $f_S(\tilde{x}) \leq q(\tilde{x})$ for each $\tilde{x} \in S$ with the observation that $g(x) = \ell^*$ is a convex function such that $g(x) \leq q(x)$ for each $x \in S$, which implies that $f_S(\tilde{x}) \geq \ell^*$ for each $\tilde{x} \in S$ by \eqref{def_conv_env}. The second relation in \eqref{conv_env_prop} follows from the first one and the convexity of $f_S(\cdot)$.

\subsection{Local Minimizers}

In this section, we review a characterization of the set of local minimizers of (QP). 

Consider an instance of (QP). Since $S$ is a polyhedral set, the constraint qualification holds at every feasible point. Therefore, if $\tilde{x}$ is a local minimizer of (QP), then there exist $\tilde{y} \in \R^m$ and $\tilde{s} \in \R^n$ such that the following KKT conditions are satisfied:
\begin{eqnarray} 
Q \tilde{x} + c - A^T \tilde{y} - \tilde{s} & = & 0, \label{kkt1}\\
A \tilde{x} & = & b, \label{kkt2}\\
\tilde{x}_j \tilde{s}_j & = & 0, \quad j = 1,\ldots,n, \label{kkt3}\\
\tilde{x} & \in & \R^n_+, \label{kkt4}\\
\tilde{s} & \in & \R^n_+. \label{kkt5}
\end{eqnarray}
We remark that $\tilde{y}$ and $\tilde{s}$ are actually the Lagrange multipliers scaled by $1/2$. 

In addition, any local minimizer $\tilde{x}$ satisfies the following second order necessary conditions:
\begin{equation} \label{2nd_ord_con}
d^T Q d \geq 0, \quad d \in C_{\tilde{x}},    
\end{equation}
where
\begin{equation} \label{def_Dx}
C_{\tilde{x}} = \left\{d \in \R^n: Ad = 0, \quad d^T (Q \tilde{x} + c) = 0, \quad d_j \geq 0, \quad j \in \mathbf{Z}(\tilde{x})\right\}. 
\end{equation}

Conversely, any $\tilde{x} \in \R^n$ that satisfies the KKT conditions \eqref{kkt1}--\eqref{kkt5} and the second order necessary conditions \eqref{2nd_ord_con} is, in fact, a local minimizer of (QP) (see, e.g.,~\cite{ref:majthay1971,ref:jiaquan1982}). Therefore, the conditions \eqref{kkt1}--\eqref{2nd_ord_con} provide a complete characterization of the set of all local minimizers of (QP).

\subsection{Convex Cones of Interest}

Let us define the following cones in $\cS^{n}$:
\begin{eqnarray}
{\cal N}^{n} & = & \left\{Y \in \cS^{n}: Y_{ij} \geq 0, ~ 1 \leq i \leq j \leq n\right\}, \label{def_N}\\
{\cal PSD}^{n} & = & \left\{Y \in \cS^{n}: Y = \sum_{j=1}^k y^j (y^j)^T, \quad y^j \in \R^n, ~ j = 1,\ldots,k \right\}, \label{def_PSD}\\
{\cal D}^{n} & = & \left\{Y \in \cS^{n}: Y \in {\cal PSD}^{n}, \quad Y \in \cN^{n} \right\}, \label{def_D}\\
{\cal CP}^{n} & = & \left\{Y \in \cS^{n+1}: Y = \sum_{j=1}^k y^j (y^j)^T, \quad y^j \in \R^n_+, ~ j = 1,\ldots,k \right\}, \label{def_C}\\
{\cal COP}^n & = & \left\{Y \in \cS^n: u^T Y u \geq 0, \quad \forall u \in \R^n_+ \right\}. \label{def_COP}\\
\end{eqnarray}

Note that ${\cal N}^{n}$ is the cone of componentwise nonnegative matrices, ${\cal PSD}^{n}$ is the cone of positive semidefinite matrices, ${\cal D}^n$ is the cone of doubly nonnegative matrices, ${\cal CP}^n$ is the cone of completely positive matrices, and ${\cal COP}^n$ is the cone of copositive matrices. Each of these cones is a proper convex cone. Furthermore, it is easy to verify that 
\begin{equation} \label{inc_rels}
{\cal CP}^{n} \subseteq {\cal D}^{n} \subseteq {\cal PSD}^{n} \subseteq {\cal PSD}^{n} + {\cal N}^n \subseteq {\cal COP}^n.
\end{equation}
Furthermore, by~\cite{ref:Diananda},
\begin{equation} \label{diananda}
{\cal CP}^n = {\cal DN}^n, \quad \textrm{and} \quad {\cal PSD}^{n} + {\cal N}^n = {\cal COP}^n \quad \textrm{if and only if} \quad n \leq 4.
\end{equation}

In addition, we define the following closed convex cone in ${\cal S}^{n+1}$:
\begin{equation} \label{def_P_cone}
{\cal P}^{n+1} = \left\{Y \in \cS^{n+1}: Y \in {\cal PSD}^{n+1}, \quad Y_{0,1:n} \in \R^n_+\right\},    
\end{equation}
i.e., ${\cal P}^{n+1}$ is the cone of positive semidefinite matrices in ${\cal S}^{n+1}$ with an additional nonnegativity constraint on the $0$th row and $0$th column. It is easy to verify that
\begin{equation} \label{inc_rels_2}
{\cal D}^{n+1} \subseteq {\cal P}^{n+1} \subseteq {\cal PSD}^{n+1}.    
\end{equation}

We finally define the following family of convex cones in ${\cal S}^{n+1}$:
\begin{equation} \label{def_cone_fam}
\mathbb{K} = \left\{{\cal K} \in \cS^{n+1}: {\cal K}~\textrm{is a closed convex cone s.t.}~{\cal CP}^{n+1} \subseteq {\cal K} \subseteq {\cal P}^{n+1}\right\}.
\end{equation}

\section{Feasibility Preserving Convex Relaxations} \label{feas_pres}

In this section, we introduce a family of convex relaxations of (QP) and establish several properties of such relaxations.

Consider the following family of problems:
\[
\begin{array}{lrrcl}
(P({\cal K})) & \ell_{\cal K} = \min & \langle \widehat{Q}, Y \rangle & & \\
& \textrm{s.t.}
 & \langle \widehat{A}, Y \rangle & = & 0\\
 & & Y_{00} & = & 1\\
 & & Y & \in & {\cal K},
\end{array}
\]
where 
\begin{eqnarray}
\widehat{Q} & = & \left[ \begin{matrix} 0 & c^T \\ c & Q \end{matrix} \right], \label{def_Qh}\\
\widehat{A} & = & \left[ \begin{matrix} b^T b & -b^T A \\ -A^T b & A^T A \end{matrix} \right] = \left[ \begin{matrix} b^T \\ -A^T \end{matrix}\right] \left[ \begin{matrix} b^T \\ -A^T  \end{matrix}\right]^T, \label{def_Ah}\\
\end{eqnarray}
and ${\cal K} \in \mathbb{K}$. The optimal value is denoted by $\ell_{\cal K} \in \R \cup \{-\infty\} \cup \{+\infty\}$, with the usual aforementioned conventions.

For any ${\cal K} \in \mathbb{K}$, it is easy to verify that $(P({\cal K}))$ is a convex relaxation of (QP) since $(P({\cal K}))$ is a convex optimization problem and ${\tilde Y} = [1;{\tilde x}][1;{\tilde x}]^T$ is a feasible solution of $(P({\cal K}))$ for any $\tilde{x} \in S$ with the same objective function value $q({\tilde x})$. Therefore, we immediately obtain
\begin{equation} \label{low_bo}
\ell_{\cal K} \leq \ell^*.    
\end{equation}
Furthermore, 
\begin{equation} \label{nested_cones}
{\cal K}_1 \in \mathbb{K}, \quad {\cal K}_2 \in \mathbb{K}, \quad {\cal K}_1 \subseteq {\cal K}_2 \quad \Longrightarrow \quad \ell_{{\cal K}_2} \leq \ell_{{\cal K}_1} \leq \ell^*.
\end{equation}
We will henceforth refer to the convex relaxation $(P({\cal K}))$ as the copositive relaxation if ${\cal K} = {\cal CP}^{n+1}$ and doubly nonnegative relaxation if ${\cal K} = {\cal D}^{n+1}$.

The next result establishes a useful property of $(P({\cal K}))$ for any ${\cal K} \in \mathbb{K}$ and forms the basis of our alternative perspective.

\begin{lemma} \label{lem1}
Let ${\cal K} \in \mathbb{K}$, where $\mathbb{K}$ is given by \eqref{def_cone_fam}. Then, 
\begin{equation} \label{feas_Y_hat}
{\tilde Y} = \left[ \begin{matrix} 1 & {\tilde x}^T \\ {\tilde x} & {\tilde X} \end{matrix} \right] \in \cS^{n+1}
\end{equation}
is a feasible solution of $(P({\cal K}))$ if and only if 
\begin{enumerate}
    \item ${\tilde x} \in S$, and
    \item there exist $\tilde{d}^j \in \R^n,~j = 1,\ldots,k$ such that $\tilde{X} = \tilde{x} \tilde{x}^T + \sum\limits_{j = 1}^k \tilde{d}^j (\tilde{d}^j)^T$, where $A \tilde{d}^j = 0$ for each $j = 1,\ldots,k$.
\end{enumerate}
\end{lemma}
\begin{proof}
Let ${\cal K} \in \mathbb{K}$ and ${\tilde Y}$ given by \eqref{feas_Y_hat} be a feasible solution of $(P({\cal K}))$. By \eqref{def_cone_fam}, we have ${\tilde x} \in \R^n_+$ and ${\tilde X} - \tilde{x} \tilde{x}^T = \tilde{D} \in {\cal PSD}^n$, i.e., $\tilde{D} = \sum_{j = 1}^k \tilde{d}^j (\tilde{d}^j)^T$ for some $\tilde{d}^j \in \R^n,~j = 1,\ldots,k$.
Furthermore, 
\[
\langle {\widehat A}, {\tilde Y} \rangle = \|b - A{\tilde x}\|^2 + \langle A^T A, \tilde{D} \rangle = 0,
\]
which implies that $\|b - A{\tilde x}\|^2 = \langle A^T A, \tilde{D} \rangle = 0$ since $A^T A \in {\cal PSD}^n$ and $\tilde{D} \in {\cal PSD}^n$. It follows that $A {\tilde x} = b$ and $A \tilde{d}^j = 0$ for each $j = 1,\ldots,k$. The converse implication can be easily established.
\end{proof}

Since $(P({\cal K}))$ is a convex relaxation of (QP), it follows that $(P({\cal K}))$ is infeasible if (QP) is infeasible. The proof of Lemma~\ref{lem1} implies the following converse result.

\begin{corollary} \label{emptyS}
If (QP) is infeasible, then $(P({\cal K}))$ is infeasible for any ${\cal K} \in \mathbb{K}$, where $\mathbb{K}$ is given by \eqref{def_cone_fam}.
\end{corollary}
\begin{proof}
Suppose that (QP) is infeasible, i.e., $S = \emptyset$. By Lemma~\ref{lem1}, if ${\tilde Y} \in \cS^{n+1}$ is feasible for $(P({\cal K}))$, then ${\tilde x} = {\tilde Y}_{0,1:n} \in S$, which is a contradiction.
\end{proof}

By Corollary~\ref{emptyS}, any convex relaxation of $(P({\cal K}))$ of (QP), where ${\cal K} \in \mathbb{K}$, will be referred to as a {\em feasibility preserving} relaxation.

For any ${\cal K} \in \mathbb{K}$, Lemma~\ref{lem1} implies that $S$ is given by the following projection of the feasible region of $(P({\cal K}))$:
\begin{equation} \label{project}
S = \left\{x \in \R^n: \exists~ Y = \left[ \begin{matrix} 1 & x^T \\ x & {X} \end{matrix} \right] \in \cS^{n+1} \textrm{ such that $Y$ is feasible for }(P({\cal K}))\right\}.
\end{equation}
This observation motivates us to define the following optimization problem parametrized by ${\tilde x} \in S$ for a given ${\cal K} \in \mathbb{K}$:
\[
\begin{array}{lrrcl}
(P({\cal K},{\tilde x})) & \ell_{\cal K}({\tilde x}) = \min & \langle \widehat{Q}, Y \rangle & & \\
 & \textrm{s.t.}
 & \langle \widehat{A}, Y \rangle & = & 0\\
 & & Y_{0,1:n} & = & {\tilde x}\\
 & & Y_{00} & = & 1\\
 & & Y & \in & {\cal K},
\end{array}
\]
where $\ell_{\cal K}({\tilde x}) \in \R \cup \{-\infty\} \cup \{+\infty\}$. Note that $(P({\cal K},{\tilde x}))$ is a constrained version of $(P({\cal K}))$. Similar to \eqref{nested_cones}, we have
\begin{equation} \label{nested_ineqs}
{\cal K}_1 \in \mathbb{K}, \quad {\cal K}_2 \in \mathbb{K}, \quad {\cal K}_1 \subseteq {\cal K}_2 \quad \Longrightarrow \quad \ell_{{\cal K}_2}({\tilde x}) \leq \ell_{{\cal K}_1}({\tilde x}), \quad \tilde{x} \in S.
\end{equation}

We have the following result.

\begin{proposition} \label{props}
Given an instance of (QP), let ${\cal K} \in \mathbb{K}$, where $\mathbb{K}$ is given by \eqref{def_cone_fam}. Then, $\ell_{\cal K}(\cdot)$ is a convex underestimator of $q(\cdot)$ on $S$, i.e., $\ell_{\cal K}(\cdot)$ is convex and
\begin{equation} \label{rel1}
\ell_{\cal K}({\tilde x}) \leq q({\tilde x}),  \quad {\tilde x} \in S.
\end{equation}
Furthermore, 
\begin{equation} \label{rel2}
\ell_{\cal K} = \min\limits_{x \in S} ~\ell_{\cal K}(x).
\end{equation}
\end{proposition}
\begin{proof}
Let ${\cal K} \in \mathbb{K}$. The convexity of $\ell_{\cal K}(\cdot)$ follows from the fact that the optimal value function of a convex optimization problem is convex as a function of the right-hand side parameter varying on a convex set (see, e.g.,~\cite[Theorem 29.1]{Roc70}). For any ${\tilde x} \in S$, note that ${\tilde Y} = [1;{\tilde x}][1;{\tilde x}]^T \in {\cal S}^{n+1}$ is a feasible solution of $(P({\cal K},{\tilde x}))$ with $\langle \widehat{Q}, {\tilde Y} \rangle = q({\tilde x})$, which establishes \eqref{rel1}. Therefore, $\ell_{\cal K}(\cdot)$ is a convex underestimator of $q(\cdot)$ on $S$.

We clearly have $\ell_{\cal K} \leq \ell_{\cal K}(\tilde{x})$ since $(P({\cal K},{\tilde x}))$ is a constrained version of $(P({\cal K}))$. If the optimal solution of $(P({\cal K}))$ is attained and given by $Y^* \in {\cal S}^{n+1}$, then $Y^*$ is also an optimal solution of $(P({\cal K},x^*))$, where $x^* = Y^*_{0,1:n}$, proving \eqref{rel2}. Otherwise, if $Y^k \in \cS^{n+1},~k = 1,2,\ldots$ is a sequence of feasible solutions of $(P({\cal K}))$ such that $\langle {\widehat Q}, Y^k \rangle \to \ell_{\cal K}$, then 
\[
\ell_{\cal K} \leq \ell_{\cal K}(x^k) \leq \langle {\widehat Q}, Y^k \rangle, \quad k = 1,2,\ldots,
\]
where $x^k = Y^k_{0,1:n}$, which implies that $\ell_{\cal K}(x^k) \to \ell_{\cal K}$, also establishing \eqref{rel2} in this case.
\end{proof}

For each ${\cal K} \in \mathbb{K}$, Proposition~\ref{props} reveals that $\ell_{\cal K}(\cdot)$ is a convex underestimator of $q(\cdot)$ on $S$. By \eqref{inc_rels}, the tightest and the weakest convex underestimators are given by ${\cal K} = {\cal CP}^{n+1}$ and ${\cal K} = {\cal P}^{n+1}$, respectively. 

We close this section by the following result of~\cite{Burer09} that outlines a useful property of the copositive relaxation (see also~\cite{ArimaKK13} for the equivalence between Burer's original formulation and the simplified formulation $(P({\cal K}))$ for $K = {\cal CP}^{n+1}$). We will later use this result to establish some properties of the convex underestimator arising from the copositive relaxation.  

\begin{lemma} \cite[Lemma 2.2]{Burer09} \label{burer_cp}
Given an instance of (QP) with $S \neq \emptyset$, let $\tilde{x} \in S$ and ${\cal K} = {\cal CP}^{n+1}$. Then, $\tilde{Y} \in \cS^{n+1}$ given by \eqref{feas_Y_hat} is a feasible solution of $(P({\cal K},{\tilde x}))$ if and only if there exist $\tilde{x}^j \in S$,~$j = 1,\ldots,k_1$, $\lambda_j \geq 0,~j = 1,\ldots,k_1$, and $\tilde{d}^j \in L,~j = 1,\ldots,k_2$, where $L$ is given by \eqref{def_L}, such that $\sum\limits_{j=1}^{k_1} \lambda_j = 1$, and 
\begin{equation} \label{burer_dec}
{\tilde Y} = \left[ \begin{matrix} 1 & {\tilde x}^T \\ {\tilde x} & {\tilde X} \end{matrix} \right] = \sum\limits_{j=1}^{k_1} \lambda_j \left[ \begin{matrix} 1 \\ \tilde{x}^j \end{matrix} \right] \left[ \begin{matrix} 1 \\ \tilde{x}^j \end{matrix} \right]^T + \sum\limits_{j=1}^{k_2} \left[ \begin{matrix} 0 \\ \tilde{d}^j \end{matrix} \right] \left[ \begin{matrix} 0 \\ \tilde{d}^j \end{matrix} \right]^T.
\end{equation}
\end{lemma}

\section{Exact and Trivial Relaxations} \label{exact_trivial}

In this section, we identify simple sufficient conditions under which the convex relaxation $(P({\cal K}))$ is exact at the one extreme and yields the trivial lower bound of $-\infty$ at the other extreme for any ${\cal K} \in \mathbb{K}$.

The following lemma identifies a sufficient condition under which the convex underestimator arising from {\em any} feasibility preserving convex relaxation coincides with the objective function $q(\cdot)$ and is therefore exact.

\begin{proposition} \label{exact_rels}
Suppose that $S \neq \emptyset$ and $Q$ is positive semidefinite on the null space of $A$, i.e., 
\begin{equation} \label{psd_NA}
Ad = 0, \quad d \in \R^n \quad \Longrightarrow \quad d^T Q d \geq 0.
\end{equation}
Then, for any ${\cal K} \in \mathbb{K}$,
\begin{equation} \label{exact_undest}
\ell_{\cal K}(x) = q(x), \quad x \in S.   
\end{equation}
Therefore, the convex relaxation $(P({\cal K}))$ is exact, i.e., $\ell_{\cal K} = \ell^*$. Furthermore, for any optimal solution $Y^* \in {\cal S}^{n+1}$ of $(P({\cal K}))$, $x^* = Y^*_{0,1:n}$ is an optimal solution of (QP).
\end{proposition}
\begin{proof}
Suppose that \eqref{psd_NA} holds and let ${\tilde x} \in S$. Let ${\tilde Y} \in {\cal S}^{n+1}$ given by \eqref{feas_Y_hat} be a feasible solution of $(P({\cal K},{\tilde x}))$. By Lemma~\ref{lem1}, 
\[
\langle {\widehat Q}, {\tilde Y} \rangle = 2c^T {\tilde x} + {\tilde x}^T Q {\tilde x} + \sum\limits_{j = 1}^k (\tilde{d}^j)^T Q \tilde{d}^j \geq q({\tilde x}),
\]
where the inequality follows from \eqref{psd_NA} since $A \tilde{d}^j = 0$ for each $j = 1,\ldots,k$ by Lemma~\ref{lem1}. By Proposition~\ref{props}, we obtain $\ell_{\cal K}({\tilde x}) = q({\tilde x})$, and $\ell_{\cal K} = \ell^*$. Finally, if $Y^* \in {\cal S}^{n+1}$ is an optimal solution of $(P({\cal K}))$, it follows from Proposition~\ref{props} and \eqref{psd_NA} that $\ell^* = \ell_{\cal K} = \langle \widehat{Q}, Y^* \rangle \geq q(x^*) = (x^*)^T Q x^* + 2 c^T x^* \geq \ell^*$, where $x^* = Y^*_{0,1:n}$, which implies that $x^*$ is an optimal solution of (QP).
\end{proof}

Under the condition \eqref{psd_NA}, it follows from Proposition~\ref{exact_rels} that solving the convex relaxation $(P({\cal K}))$ not only yields the optimal value $\ell^*$ of (QP) but also an optimal solution of (QP) if it is attained.

We remark that the condition \eqref{psd_NA} is clearly satisfied if $q(\cdot)$ is a convex function. Under this convexity assumption on $q(\cdot)$, a similar result for the doubly nonnegative relaxation follows from \cite[Lemma 2.7]{Kim2020JG}. Similarly, under the same condition \eqref{psd_NA}, G{\"o}kmen and Y{\i}ld{\i}r{\i}m~\cite[Proposition 4.3]{GY2020} establish the exactness of the doubly nonnegative relaxation for the special case of standard quadratic programs. Therefore, Proposition~\ref{exact_rels} subsumes and extends these results in two directions. First, the exactness result holds under the slightly weaker condition \eqref{psd_NA}. Second, it holds for any ${\cal K} \in \mathbb{K}$, including the weaker 
convex relaxation arising from ${\cal K} = {\cal P}^{n+1}$. 

It is also worth noticing that Propositions~\ref{props} and \ref{exact_rels} readily imply the convexity of $q(\cdot)$ on $S$.

We now focus on the other extreme, i.e., the case in which the convex relaxation $(P({\cal K}))$ yields no useful information on (QP). To that end, our first result establishes that the weakest convex relaxation arising from ${\cal K} = {\cal P}^{n+1}$ yields a trivial convex underestimator if the condition \eqref{psd_NA} fails.

\begin{proposition} \label{trivial_lb_P}
Suppose that $S \neq \emptyset$ and \eqref{psd_NA} fails, i.e., there exists $\tilde{d} \in \R^n$ such that $A \tilde{d} = 0$ and $\tilde{d}^T Q \tilde{d} < 0$. Then, for ${\cal K} = {\cal P}^{n+1}$, $\ell_{\cal K} (\tilde{x}) = -\infty$ for all $\tilde{x} \in S$, and $\ell_{\cal K} = -\infty$.
\end{proposition}
\begin{proof}
Let ${\cal K} = {\cal P}^{n+1}$ and ${\tilde x} \in S$. Let ${\tilde d} \in \R^n$ be such that $A {\tilde d} = 0$ and ${\tilde d}^T Q {\tilde d} < 0$. By Lemma~\ref{lem1}, for any $\lambda \geq 0$, 
\[
{\tilde Y}(\lambda) = \left[ \begin{matrix} 1 \\ {\tilde x} \end{matrix} \right] \left[ \begin{matrix} 1 \\ {\tilde x} \end{matrix} \right]^T + \lambda \left[ \begin{matrix} 0 \\ {\tilde d} \end{matrix} \right] \left[ \begin{matrix} 0 \\ {\tilde d} \end{matrix} \right]^T
\]
is a feasible solution of $(P({\cal K},{\tilde x}))$ and $\langle {\tilde Q}, {\tilde Y} \rangle = q({\tilde x}) + \lambda {\tilde d}^T Q {\tilde d} \to -\infty$ as $\lambda \to +\infty$, establishing that $\ell_{\cal K} (\tilde{x}) = -\infty$ for all $\tilde{x} \in S$. The last assertion follows from Proposition~\ref{props}.
\end{proof}

By Propositions~\ref{exact_rels} and \ref{trivial_lb_P}, it follows that the convex relaxation arising from ${\cal K} = {\cal P}^{n+1}$ is exact under the condition \eqref{psd_NA} and provides no useful information otherwise. It is also worth pointing out that Proposition~\ref{trivial_lb_P} also holds for the weaker convex relaxation arising from ${\cal K} = {\cal PSD}^{n+1}$. We remark, however, that the corresponding convex relaxation is not a feasibility preserving relaxation.

Under a stronger condition, we next present a result similar to Proposition~\ref{trivial_lb_P} that holds for any ${\cal K} \in \mathbb{K}$.

\begin{proposition} \label{neg_curvature}
Suppose that $S \neq \emptyset$ and $L_\infty \neq \emptyset$, where $L_\infty$ is given by \eqref{def_Linf}. Then, for any ${\cal K} \in \mathbb{K}$, $\ell_{\cal K}(\tilde{x}) = -\infty$ for any $\tilde{x} \in S$, and $\ell_{\cal K} = \ell^* = -\infty$.
\end{proposition}
\begin{proof}
By \eqref{nested_ineqs}, it suffices to prove the assertion for ${\cal K} = {\cal CP}^{n+1}$. Similar to the proof of Proposition~\ref{trivial_lb_P}, let ${\tilde x} \in S$ and $\tilde{d} \in L_\infty$. For any $\lambda \geq 0$, ${\tilde Y}(\lambda) = [1;{\tilde x}][1;{\tilde x}]^T + \lambda [0;{\tilde d}][0;{\tilde d}]^T$ is a feasible solution of $(P({\cal K},{\tilde x}))$ and $\langle {\tilde Q}, {\tilde Y} \rangle = q({\tilde x}) + \lambda {\tilde d}^T Q {\tilde d} \to -\infty$ as $\lambda \to +\infty$, establishing $\ell_{\cal K}(\tilde{x}) = -\infty$ for any $\tilde{x} \in S$. By Proposition~\ref{props} and Lemma~\ref{unb_qp}, we obtain $\ell_{\cal K} = \ell^* = -\infty$.
\end{proof}

\section{Bounded and Unbounded Feasible Regions} \label{bou_vs_unbou}

In this section, we present several results for feasibility preserving convex relaxations of instances of (QP) with bounded and unbounded feasible regions. 

First, we focus on the relation between the boundedness of the feasible region $S$ of (QP) and that of the feasible region of $(P({\cal K}))$ for ${\cal K} \in \mathbb{K}$. By \eqref{project}, it is clear that the boundedness of the latter set implies the boundedness of the former, and the unboundedess of the former set implies the unboundedness of the latter set. However, the converse implications, in general, do not hold for each ${\cal K} \in \mathbb{K}$. For instance, even if $S$ is a nonempty polytope, Proposition~\ref{trivial_lb_P} reveals that the feasible region of $(P({\cal K}))$ can be unbounded for ${\cal K} = {\cal P}^{n+1}$ if the condition \eqref{psd_NA} fails. 

Let us next define the following subset of $\mathbb{K}$:
\begin{equation} \label{def_K_star}
\mathbb{K}^* = \left\{{\cal K} \in \mathbb{K}: {\cal K} \subseteq {\cal D}^{n+1}\right\}.
\end{equation}
Under the stronger assumption that ${\cal K} \in \mathbb{K}^*$, 
the aforementioned converse implications were established in \cite[Lemma 2]{KimKT16}, i.e.,  
the boundedness of $S$ implies the boundedness of the feasible region of any feasibility preserving convex relaxation $(P({\cal K}))$. We state this result in a slightly different form below and give a slightly different proof that highlights the relation between the recession cone of $S$ and that of the feasible region of $(P({\cal K}))$. 

\begin{proposition} \label{ub_feas_reg}
Let ${\cal K} \in \mathbb{K}^*$, where $\mathbb{K}^*$ is given by \eqref{def_K_star}. Then, $S$ is unbounded if and only if the feasible region of $(P({\cal K}))$ is unbounded.
\end{proposition}
\begin{proof}
Let ${\cal K} \in \mathbb{K}^*$. By the preceding discussion, if $S$ is unbounded, then the feasible region of $(P({\cal K}))$ is unbounded. Conversely, suppose that the feasible region of $(P({\cal K}))$ is unbounded. By Lemma~\ref{lem1}, $S \neq \emptyset$. Then, by \cite[Theorem 8.4]{Roc70}, the recession cone of the feasible region of $(P({\cal K}))$ is nonempty, i.e., there exists $D \in {\cal K}$ such that $D \neq 0$, $D_{00} = 0$, which implies that $D_{0,1:n} = 0$ since $D \in {\cal PSD}^{n+1}$, and $\langle \widehat{A}, D \rangle = 0$. Therefore, $D$ is given by
\[
D = \left[ \begin{matrix} 0 & 0 \\ 0 & {\tilde D} \end{matrix} \right],
\]
where $\tilde{D} \in {\cal D}^n$ and $\langle A^T A, \tilde{D} \rangle = 0$. Using a similar argument as in the proof of Lemma~\ref{lem1}, it follows that $\tilde{D} = \sum\limits_{j = 1}^k \tilde{d}^j (\tilde{d}^j)^T$, where $A \tilde{d}^j = 0$ for each $j = 1,\ldots,k$. Since $\tilde{D} \in \cN^n \backslash \{0\}$, we have ${\tilde d} = \tilde{D} e \in \R^n_+ \backslash \{0\}$, i.e., ${\tilde d} = \sum\limits_{j = 1}^k (e^T \tilde{d}^j) \tilde{d}^j$. Therefore, $A {\tilde d} = 0$, which implies that ${\tilde d} \in L$. Therefore, $S$ is unbounded since ${\tilde d} \neq 0$. 
\end{proof}

Under the assumption that ${\cal K} \in \mathbb{K}^*$, Proposition~\ref{ub_feas_reg} reveals that the corresponding convex relaxations not only preserve feasibility but also boundedness of the feasible region.

\subsection{Bounded Feasible Region}

In this section, we focus on instances of (QP) for which $S$ is a nonempty polytope. 

We start with the following immediate corollary of Proposition~\ref{ub_feas_reg}, which was first proved in~\cite{KimKT16}.

\begin{corollary}\cite[Lemma 2]{KimKT16} \label{finite_low_bound}
Let $S$ be a nonempty polytope and let ${\cal K} \in \mathbb{K}^*$, where $\mathbb{K}^*$ is given by \eqref{def_K_star}. Then, $\ell_{\cal K} > -\infty$. Furthermore, the set of optimal solutions of $(P({\cal K}))$ is nonempty. 
\end{corollary}

We next identify a useful property of the set of feasible solutions of $P({\cal K})$ for ${\cal K} \in \mathbb{K}^*$. We will then use this property to establish a result about local minimizers of (QP).

\begin{proposition} \label{feas_k_star}
Let ${\cal K} \in \mathbb{K}^*$ and consider any instance of (QP), where $S$ is a nonempty polytope. Let $\tilde{x} \in S$ and let $\tilde{Y} \in \cS^{n+1}$  be a feasible solution of $P({\cal K},\tilde{x})$ given by \eqref{feas_Y_hat}. Then, for any decomposition of ${\tilde X} = \tilde{Y}_{1:n,1:n}$ given by $\tilde{X} = \tilde{x} \tilde{x}^T + \sum\limits_{j = 1}^k \tilde{d}^j (\tilde{d}^j)^T$ (cf.~Lemma~\ref{lem1}), we have
\begin{equation} \label{td_Z_zero}
A \tilde{d}^j = 0, \quad \tilde{d}^j_i = 0, \quad j = 1,\ldots,k,~ i \in \mathbf{Z}(\tilde{x}),
\end{equation}
where $\mathbf{Z}(\tilde{x})$ is given by \eqref{def_Zx}.
\end{proposition}
\begin{proof}
Suppose that $S$ is a nonempty polytope. Let $\tilde{x} \in S$ and let ${\cal K} \in \mathbb{K}^*$. Let $\tilde{Y} \in \cS^{n+1}$ be a feasible solution of $P({\cal K},\tilde{x})$ given by \eqref{feas_Y_hat} and let ${\tilde X} = \tilde{Y}_{1:n,1:n}$. By  Lemma~\ref{lem1}, $\tilde{X} = \tilde{x} \tilde{x}^T + \sum\limits_{j = 1}^k \tilde{d}^j (\tilde{d}^j)^T$, where $A \tilde{d}^j = 0$ for each $j = 1,\ldots,k$. Let us define $\mathbf{P} = \mathbf{P}(\tilde{x})$ and $\mathbf{Z} = \mathbf{Z}(\tilde{x})$, where $\mathbf{P}(\tilde{x})$ and $\mathbf{Z}(\tilde{x})$ are given by \eqref{def_Px} and \eqref{def_Zx}, respectively. Then, by permuting the rows and columns of ${\tilde X}$ if necessary, we obtain
\[
\tilde{U} = \left[ \begin{matrix} \tilde{X}_{\mathbf{P}\mathbf{P}} & \tilde{X}_{\mathbf{P}\mathbf{Z}} \\  \tilde{X}_{\mathbf{Z}\mathbf{P}} & \tilde{X}_{\mathbf{Z}\mathbf{Z}} \end{matrix} \right] = \left[ \begin{matrix} \tilde{x}_{\mathbf{P}} \tilde{x}_{\mathbf{P}}^T + \sum\limits_{j = 1}^k \tilde{d}^j_{\mathbf{P}} (\tilde{d}^j_{\mathbf{P}})^T & \sum\limits_{j = 1}^k \tilde{d}^j_{\mathbf{P}} (\tilde{d}^j_{\mathbf{Z}})^T \\
\sum\limits_{j = 1}^k \tilde{d}^j_{\mathbf{Z}} (\tilde{d}^j_{\mathbf{P}})^T & \sum\limits_{j = 1}^k \tilde{d}^j_{\mathbf{Z}} (\tilde{d}^j_{\mathbf{Z}})^T \end{matrix} \right],
\]
where we used $\tilde{x}_{\mathbf{Z}} = 0$. Since $\tilde{U} \in {\cal D}^n$, it follows that
\[
\tilde{d} = \left[ \begin{matrix} \tilde{X}_{\mathbf{P}\mathbf{P}} & \tilde{X}_{\mathbf{P}\mathbf{Z}} \\  \tilde{X}_{\mathbf{Z}\mathbf{P}} & \tilde{X}_{\mathbf{Z}\mathbf{Z}} \end{matrix} \right] \left[ \begin{matrix} 0 \\ e \end{matrix} \right] = \left[ \begin{matrix} \sum\limits_{j = 1}^k \left(  (\tilde{d}^j_{\mathbf{Z}})^T e \right) \tilde{d}^j_{\mathbf{P}} \\ \sum\limits_{j = 1}^k \left(  (\tilde{d}^j_{\mathbf{Z}})^T e \right) \tilde{d}^j_{\mathbf{Z}} \end{matrix} \right] = \left[ \begin{matrix} \tilde{d}_\mathbf{P} \\ \tilde{d}_\mathbf{Z} \end{matrix} \right] \geq 0.
\]
Therefore, 
\begin{eqnarray*}
A \tilde{d} & = & A_{1:m,\mathbf{P}} \tilde{d}_\mathbf{P} + A_{1:m,\mathbf{Z}} \tilde{d}_\mathbf{Z}\\
 & = & \sum\limits_{j = 1}^k \left(  (\tilde{d}^j_{\mathbf{Z}})^T e \right) A_{1:m,\mathbf{P}} \tilde{d}^j_{\mathbf{P}} + \sum\limits_{j = 1}^k \left(  (\tilde{d}^j_{\mathbf{Z}})^T e \right) A_{1:m,\mathbf{Z}} \tilde{d}^j_{\mathbf{Z}} \\
 & = & \sum\limits_{j = 1}^k \left(  (\tilde{d}^j_{\mathbf{Z}})^T e \right) \left( A_{1:m,\mathbf{P}} \tilde{d}^j_{\mathbf{P}} + A_{1:m,\mathbf{Z}} \tilde{d}^j_{\mathbf{Z}} \right) \\
 & = & \sum\limits_{j = 1}^k \left(  (\tilde{d}^j_{\mathbf{Z}})^T e \right) A \tilde{d}^j\\
 & = & 0.
\end{eqnarray*}
Since $S$ is bounded, $A \tilde{d} = 0$, and $\tilde{d} \geq 0$, it follows that $\tilde{d} = 0$. Therefore,
\[
\tilde{d}_\mathbf{Z} = \sum\limits_{j = 1}^k \left(  (\tilde{d}^j_{\mathbf{Z}})^T e \right) \tilde{d}^j_{\mathbf{Z}} = \left( \sum\limits_{j = 1}^k \tilde{d}^j_{\mathbf{Z}} (\tilde{d}^j_{\mathbf{Z}})^T \right) e = 0.
\]
Since $\tilde{U} \in {\cal D}^n$, it follows that $\sum\limits_{j = 1}^k \tilde{d}^j_{\mathbf{Z}} (\tilde{d}^j_{\mathbf{Z}})^T \in {\cal N}^{|\mathbf{Z}|}$, which, together with the last equality, implies that $\tilde{d}^j_{\mathbf{Z}} = 0$ for each $j = 1,\ldots,k$. The assertion follows.
\end{proof}

This technical result enables us to establish the following property.

\begin{proposition} \label{loc_min_prop}
Let ${\cal K} \in \mathbb{K}^*$ and consider any instance of (QP), where $S$ is a nonempty polytope. Then, for any local minimizer $\tilde{x} \in S$ of (QP), we have
\begin{equation} \label{loc_min_obj}
\ell_{\cal K}(\tilde{x}) = q({\tilde x}).
\end{equation}
In particular, for any global minimizer $x^* \in S$ of (QP), $\ell_{\cal K}({x}^*) = q({x}^*) = \ell^*$.
\end{proposition}
\begin{proof}
Let $S$ be a nonempty polytope, ${\cal K} \in \mathbb{K}^*$, and let $\tilde{x} \in S$ be a local minimizer of (QP). Let us define $\mathbf{P} = \mathbf{P}(\tilde{x})$ and $\mathbf{Z} = \mathbf{Z}(\tilde{x})$, where $\mathbf{P}(\tilde{x})$ and $\mathbf{Z}(\tilde{x})$ are given by \eqref{def_Px} and \eqref{def_Zx}, respectively. Let $\tilde{Y}$ given by \eqref{feas_Y_hat} be a feasible solution of $(P({\cal K},{\tilde x}))$. Then, by Lemma~\ref{lem1}, ${\tilde X} = \tilde{Y}_{1:n,1:n} = \tilde{x} \tilde{x}^T + \sum\limits_{j = 1}^k \tilde{d}^j (\tilde{d}^j)^T$, where $A \tilde{d}^j = 0$ for each $j = 1,\ldots,k$. By Proposition~\ref{feas_k_star}, $\tilde{d}^j_i = 0$ for each $j = 1,\ldots,k$ and each $i \in \mathbf{Z}$. Since $\tilde{x}$ is a local minimizer, there exist $\tilde{y} \in \R^m$ and $\tilde{s} \in \R^n$ such that \eqref{kkt1}--\eqref{kkt5} are satisfied. Therefore, for each $j = 1,\ldots,k$,
\[
(\tilde{d}^j)^T (Q \tilde{x} + c) = (\tilde{d}^j)^T \left( A^T \tilde{y} + \tilde{s} \right) = \tilde{y}^T \left( A \tilde{d}^j \right) + (\tilde{d}^j)^T \tilde{s} = 0 + \sum\limits_{i \in \mathbf{P}} \tilde{d}^j_i \tilde{s}_i + \sum\limits_{i \in \mathbf{Z}} \tilde{d}^j_i \tilde{s}_i = 0,
\]
where we used \eqref{kkt1} in the first equality, $A \tilde{d}^j = 0$ in the third equality, $\tilde{s}_i = 0$ for $i \in \mathbf{P}$ by \eqref{kkt3}, and $\tilde{d}^j_i = 0$ for $i \in \mathbf{Z}$ in the last one. By \eqref{def_Dx}, we have $\tilde{d}^j \in C_{\tilde{x}}$ for each $j = 1,\ldots,k$. Since $\tilde{x}$ is a local minimizer, $(\tilde{d}^j)^T Q \tilde{d}^j \geq 0$ for each $j = 1,\ldots,k$ by \eqref{2nd_ord_con}. Therefore, 
\[
\langle \widehat{Q}, \tilde{Y} \rangle = q({\tilde x}) + \sum\limits_{j=1}^k (\tilde{d}^j)^T Q \tilde{d}^j \geq q({\tilde x})
\]
for any feasible solution $\tilde{Y}$ of $(P({\cal K},{\tilde x}))$. Therefore, $\ell_{\cal K}(\tilde{x}) \geq q({\tilde x})$. Combining this inequality with \eqref{rel1} yields \eqref{loc_min_obj}. The second assertion simply follows from the fact that any global minimizer $x^*$ of (QP) is also a local minimizer.
\end{proof}

By Proposition~\ref{loc_min_prop}, the convex underestimator $\ell_{\cal K}(\cdot)$ agrees with $q(\cdot)$ at each local minimizer of (QP) for any ${\cal K} \in \mathbb{K}^*$. By Proposition~\ref{props},
\[
\ell_{\cal K}(x) \leq f_S(x) \leq q(x), \quad x \in S,
\]
where $f_S(\cdot)$ denotes the convex envelope of $q(\cdot)$ on $S$ given by \eqref{def_conv_env}. 
Therefore, by Proposition~\ref{loc_min_prop}, for any local minimizer $\tilde{x} \in S$ of (QP), $\ell_{\cal K}(\tilde{x}) = f_S({\tilde x}) = q({\tilde x})$ for any ${\cal K} \in \mathbb{K}^*$.

Our next result establishes a desirable property of the copositive relaxation of (QP) arising from ${\cal K} = {\cal CP}^{n+1}$.

\begin{proposition} \label{conv_env_bou}
Let $S$ be a nonempty polytope and let ${\cal K} = {\cal CP}^{n+1}$. Then, the convex underestimator $\ell_{\cal K}(\cdot)$ is the convex envelope of $q(\cdot)$ on $S$, i.e.,
\[
\ell_{\cal K}(x) = f_S(x), \quad x \in S,
\]
where $f_S(\cdot)$ denotes the convex envelope of $q(\cdot)$ on $S$ given by \eqref{def_conv_env}. Therefore, $\ell_{\cal K} = \ell^*$. In addition, $\ell_{\cal K}({\tilde x}) = \ell^*$ if and only if ${\tilde x} \in \textrm{conv}(S^*)$, where $S^*$ is given by \eqref{def_S_star}.
\end{proposition}
\begin{proof}
Let $\tilde{x} \in S$. Since $S$ is bounded, the feasible region of $(P({\cal K},{\tilde x}))$ is bounded by Proposition~\ref{ub_feas_reg}. Therefore, let $\tilde{Y} \in \cS^{n+1}$ given by \eqref{feas_Y_hat} be an optimal solution of $(P({\cal K},{\tilde x}))$. By Lemma~\ref{burer_cp}, $\tilde{Y}$ admits a decomposition given by \eqref{burer_dec}. Since $L = \{0\}$, we obtain $\tilde{d}^j = 0$ for each $j = 1,\ldots,k_2$. Therefore, 
\begin{eqnarray*}
f_S({\tilde x}) & \geq & \ell_{\cal K}({\tilde x}) \\
 & = & \langle \widehat{Q}, \tilde{Y} \rangle \\
 & = & \sum\limits_{j = 1}^{k_1} \lambda_j q({\tilde x}^j)\\
 & \geq & \sum\limits_{j = 1}^{k_1} \lambda_j f_S({\tilde x}^j)\\
 & \geq & f_S\left( \sum\limits_{j = 1}^{k_1} \lambda_j {\tilde x}^j \right) \\
 & = & f_S({\tilde x}),
 \end{eqnarray*}
 where we used \eqref{def_conv_env} and Proposition~\ref{props} in the first line, \eqref{burer_dec} in the third line, \eqref{def_conv_env} in the fourth line, the convexity of $f_S(\cdot)$ on the fifth line, and \eqref{burer_dec} again in the last line. Therefore, $\ell_{\cal K}(\tilde{x}) = f_S(\tilde{x})$ for each $\tilde{x} \in S$, which implies that $\ell_{\cal K}(\cdot)$ is the convex envelope of $q(\cdot)$ on $S$. We therefore obtain $\ell_{\cal K} = \ell^*$ by Proposition~\ref{props} and \eqref{conv_env_prop}.
 
 Finally, if ${\tilde x} \in \textrm{conv}(S^*)$, then
 $\ell_{\cal K}({\tilde x}) = \ell^*$ by \eqref{conv_env_prop}. 
%
Conversely, if ${\tilde x} \not \in \textrm{conv}(S^*)$, then let $\tilde{Y} \in \cS^{n+1}$ given by \eqref{feas_Y_hat} be an optimal solution of $(P({\cal K},{\tilde x}))$. Consider the decomposition of $\tilde{Y}$ given by \eqref{burer_dec}. It follows that there exists $j^* \in \{1,\ldots,k_1\}$ such that $\lambda_{j^*} > 0$ and $\tilde{x}^{j^*} \not \in S^*$. Therefore, $q(\tilde{x}^{j^*}) > \ell^*$, which implies that
 \[
\ell_{\cal K}({\tilde x}) = \sum\limits_{j = 1}^{k_1} \lambda_j q({\tilde x}^j) > \ell^*,
\]
proving the last assertion.
\end{proof}
 
We remark that Proposition~\ref{conv_env_bou} implies the well-known result of Burer~\cite{Burer09} for instances of (QP) with a bounded feasible region. In addition to being an exact relaxation, our perspective reveals that the copositive relaxation implicitly gives rise to the convex envelope of $q(\cdot)$ on $S$. 

We note that the convex envelope of $q(\cdot)$ on a polytope was also characterized by an optimization problem in~\cite[Theorem 1]{Anstreicher12}. Together with~\cite[Corollary 2.5]{Burer09}, this result yields another proof of Proposition~\ref{conv_env_bou}. However, we think that our perspective based on the optimization problem $(P({\cal K},{\tilde x}))$ clearly pinpoints the role played by the convex underestimator arising from the copositive relaxation on the exactness of this relaxation.

We close this section by the following corollary that extends Proposition~\ref{conv_env_bou} in lower dimensions. 

\begin{corollary} \label{conv_env_bou_small_dim}
Let $S$ be a nonempty polytope and let ${\cal K} \in \mathbb{K}^*$. For any $n \leq 4$, the convex underestimator $\ell_{\cal K}(\cdot)$ is the convex envelope of $q(\cdot)$ on $S$, i.e.,
\[
\ell_{\cal K}(x) = f_S(x), \quad x \in S,
\]
where $f_S(\cdot)$ denotes the convex envelope of $q(\cdot)$ on $S$ given by \eqref{def_conv_env}. Therefore, $\ell_{\cal K} = \ell^*$. In addition, $\ell_{\cal K}({\tilde x}) = \ell^*$ if and only if ${\tilde x} \in \textrm{conv}(S^*)$, where $S^*$ is given by \eqref{def_S_star}. 
\end{corollary}
\begin{proof}
For $n \leq 3$, the result is an immediate corollary of Proposition~\ref{conv_env_bou} and \eqref{diananda} since ${\cal K} = {\cal CP}^{n+1}$ for any ${\cal K} \in \mathbb{K}^*$ and any $n \leq 3$.

Let $n = 4$. We follow a similar argument as in~\cite[Section 3]{Burer09}. Since $S$ is a nonempty polytope, there exists $\tilde{y} \in \R^m$ such that $A^T \tilde{y} = a \in \R^n_+$ and $b^T \tilde{y} = 1$. This fact can be easily verified by maximizing, for instance, $f^T x$ on $S$, where $f \in \R^n_+$, using the boundedness of $S$, and linear programming duality. Therefore, $\tilde{x} \in S$ implies $a^T \tilde{x} = 1$. 

Let ${\cal K} \in \mathbb{K}^*$. If ${\cal K} = {\cal CP}^{5}$, then the result follows from Proposition~\ref{conv_env_bou}. Suppose that ${\cal K} \neq {\cal CP}^{5}$. Let $\tilde{x} \in S$ and $\tilde{Y} \in \cS^{n+1}$ given by \eqref{feas_Y_hat} be a feasible solution of $(P({\cal K},{\tilde x}))$.  
By Lemma~\ref{lem1}, $\tilde{X} = \tilde{x} \tilde{x}^T + \tilde{D}$, where $\tilde{D} = \sum\limits_{j = 1}^k \tilde{d}^j (\tilde{d}^j)^T$ and $A \tilde{d}^j = 0$ for each $j = 1,\ldots,k$. Therefore,
\[
A \tilde{X} A^T = A \tilde{x} \tilde{x}^T A^T + A \left( \sum\limits_{j = 1}^k \tilde{d}^j (\tilde{d}^j)^T \right) A^T = b b^T,
\]
which implies that $\tilde{y}^T A \tilde{X} A^T \tilde{y} = a^T \tilde{X} a = (\tilde{y}^T b)^2 = 1$. Furthermore, 
\[
\tilde{X} a = \left( a^T \tilde{x} \right) \tilde{x}  + \tilde{D} a = \tilde{x} + \tilde{D} a.
\]
Therefore, $1 = a^T \tilde{X} a = a^T \tilde{x} + a^T \tilde{D} a = 1 + a^T \tilde{D} a$, which implies that $a^T \tilde{D} a = 0$. Since $\tilde{D} \in {\cal PSD}^{n}$, it follows that $\tilde{D} a = 0$. Therefore, we obtain $\tilde{X} a = \tilde{x}$, and
\[
\left[ \begin{matrix} a^T \\ I \end{matrix} \right] \tilde{X} \left[ \begin{matrix} a^T \\ I \end{matrix} \right]^T = \left[ \begin{matrix} a^T \tilde{X} a & a^T \tilde{X} \\ \tilde{X} a & \tilde{X} \end{matrix} \right] = \left[ \begin{matrix} 1 & \tilde{x}^T \\ \tilde{x} & \tilde{X} \end{matrix} \right] = \tilde{Y}.
\]
Furthermore, since $\tilde{Y} \in {\cal K}$ and ${\cal K} \in \mathbb{K}^*$, it follows that $\tilde{X} \in {\cal D}^n$, which implies that $\tilde{X} \in {\cal CP}^n$ by \eqref{diananda}. Therefore, there exists $U \in \R^{n \times k}$ such that $U$ is componentwise nonnegative and $\tilde{X} = U U^T$. Combining this with the last equation above, we obtain
\[
\tilde{Y} = \left[ \begin{matrix} a^T \\ I \end{matrix} \right] U U^T \left[ \begin{matrix} a^T \\ I \end{matrix} \right]^T \in {\cal CP}^{n+1},
\]
since $a \in \R^n_+$. Therefore, the feasible region of $(P({\cal K},{\tilde x}))$ is the same as that of the corresponding copositive relaxation. The result follows from Proposition~\ref{conv_env_bou}.
\end{proof}

Corollary~\ref{conv_env_bou_small_dim} implies that the value of the convex envelope $f_S(\tilde{x})$, where $\tilde{x} \in S$, can be computed within any precision in polynomial time for any instance of (QP) with a nonempty and bounded feasible region if $n \leq 4$.

\subsection{Unbounded Feasible Region}

In this section, we focus on instances of (QP) with an unbounded feasible region. 

By Proposition~\ref{neg_curvature}, recall that the convex underestimator arising from each feasibility preserving convex relaxation yields the trivial lower bound of $-\infty$ under the assumption that $L_\infty \neq \emptyset$, where $L_\infty$ is given by \eqref{def_Linf}. Under this assumption, recall that (QP) is also unbounded below by Lemma~\ref{unb_qp}. Therefore, we obtain $\ell_{\cal K} = \ell^* = -\infty$ if $L_\infty \neq \emptyset$.

Therefore, in this section, we make the following assumption.

\begin{assumption} \label{assum2}
We assume that $L_\infty = \emptyset$, i.e.,
\[
d \in L \quad \Longrightarrow \quad d^T Q d \geq 0.
\]
\end{assumption}

By Lemma~\ref{unb_qp}, we recall that (QP) may still be unbounded below under Assumption~\ref{assum2}.

By Corollary~\ref{finite_low_bound}, we have $-\infty < \ell_{\cal K} \leq \ell^*$ for any ${\cal K} \in \mathbb{K}^*$ if $S$ is nonempty and bounded. The following example, inspired from \cite{burer2009difference}, illustrates that a similar result does not hold in general for instances of (QP) with an unbounded feasible region even if $\ell^* > -\infty$. 

\begin{example} \label{unb_example}
Consider the following instance of (QP) given by
\begin{eqnarray*}
{\bar Q} & = & \left[ \begin{matrix} 1 & -1 & 1 & 1 & -1\\-1 & 1 & -1 & 1 & 1\\1 & -1 & 1 & -1 & 1 \\ 1 & 1 & -1 & 1 & -1\\-1 & 1 & 1 & -1 & 1 \end{matrix} \right], \\ 
{\bar c} & = & \left[ \begin{matrix} 1 & 1 & 1 & 1 & 1 \end{matrix} \right]^T, \\
{\bar A} & = & \left[ \begin{matrix}-12 & 8 & -3 & 4 & 4 \end{matrix} \right],\\
{\bar b} & = & 9.
\end{eqnarray*}
Then, $S \neq \emptyset$ since ${\tilde x} = \left[ \begin{matrix}0 & 1 & 1 & 1 & 0 \end{matrix} \right]^T \in S$. Note that ${\bar Q} \in {\cal COP}^5$ since it is the well-known Horn matrix (see, e.g.,~\cite{hall1963copositive}). Since ${\bar c} \in \R^n_+$, it follows that $q(x) = x^T {\bar Q} x + 2 {\bar c}^T x \geq 0$ for any $x \in \R^n_+$. It follows that $q(\cdot)$ is bounded below on $S$. Therefore, $-\infty < \ell^* < +\infty$.

Let us consider the convex relaxation of (QP) obtained from ${\cal K} = {\cal D}^6$. We claim that $\ell_{\cal K}({\tilde x}) = -\infty$ for each ${\tilde x} \in S$. Indeed, for any ${\tilde x} \in S$, consider
\[
{\tilde Y}(\lambda) = \left[ \begin{matrix} 1 \\ {\tilde x} \end{matrix} \right] \left[ \begin{matrix} 1 \\ {\tilde x} \end{matrix} \right]^T + \lambda \left[ \begin{matrix} 0 & 0 \\ 0 & \tilde{D} \end{matrix} \right],
\]
where
\[
\tilde{D} = \left[ \begin{matrix} 1 \\ 1 \\ 0 \\ 0 \\ 1 \end{matrix} \right] \left[ \begin{matrix} 1 \\ 1 \\ 0 \\ 0 \\ 1 \end{matrix} \right]^T + \left[ \begin{matrix} 0 \\ 1 \\ 4 \\ 1 \\ 0 \end{matrix} \right] \left[ \begin{matrix} 0 \\ 1 \\ 4 \\ 1 \\ 0 \end{matrix} \right]^T + \left[ \begin{matrix} 0 \\ 1 \\ 0 \\ -1 \\ -1 \end{matrix} \right] \left[ \begin{matrix} 0 \\ 1 \\ 0 \\ -1 \\ -1 \end{matrix} \right]^T = \left[ \begin{matrix} 7 & 4 & 0 & 0 & 4 \\4 & 7 & 4 & 0 & 0 \\0 & 4 & 7 & 4 & 0\\0 & 0 & 4 & 7 & 4\\4 & 0 & 0 & 4 & 7 \end{matrix} \right].
\]
Note that $\tilde{D} \in {\cal D}^5$, $\langle {\bar A}^T {\bar A}, \tilde{D} \rangle = 0$, and $\langle {\bar Q}, \tilde{D} \rangle = -5 < 0$. Therefore, by Lemma~\ref{lem1}, ${\tilde Y}(\lambda)$ is a feasible solution of $(P({\cal K},{\tilde x}))$ for any $\lambda \geq 0$, and $\langle {\widehat Q}, {\tilde Y} \rangle = q({\tilde x}) + \lambda \langle {\bar Q}, \tilde{D} \rangle \to -\infty$ as $\lambda \to +\infty$, establishing the assertion. It follows that $\ell_{\cal K} = -\infty$ by Proposition~\ref{props}, whereas $\ell^*$ is finite.
\end{example}

Our next result generalizes the construction in Example~\ref{unb_example} to any $n \geq 5$.

\begin{proposition} \label{gen_unb_dnn}
Let ${\cal K} = {\cal D}^n$. For any $n \geq 5$, there exists an instance of (QP) such that $-\infty < \ell^* < +\infty$, whereas $\ell_{\cal K}({\tilde x}) = -\infty$ for each ${\tilde x} \in S$ and $\ell_{\cal K} = -\infty$. 
\end{proposition}
\begin{proof}
For $n = 5$, Example~\ref{unb_example} constitutes such an instance. For any $n \geq 6$, we can construct an instance of (QP) as follows:
\begin{eqnarray*}
Q & = & \left[ \begin{matrix} \bar{Q} & B \\ B^T & M \end{matrix} \right], \\
{c} & = & \left[ \begin{matrix} \bar{c} \\ f \end{matrix} \right], \\
{A} & = & \left[ \begin{matrix} \bar{A} & F \end{matrix} \right], \\
 b & = & {\bar b},
\end{eqnarray*}
where $\bar{Q}, \bar{c}, \bar{A}$ and $\bar{b}$ are defined as in Example~\ref{unb_example}, and $B, M, f$ and $F$ can be chosen arbitrarily subject to the following constraints: $B \in \R^{5 \times (n-5)}$ is componentwise nonnegative, $M \in {\cal COP}^{n-5}$, $f \in \R^{n-5}_+$, and $F \in \R^{1 \times (n-5)}$. Note that $S \neq \emptyset$ since ${\widehat x} = [{\tilde x};0] \in S$, where ${\tilde x}$ is defined as in Example~\ref{unb_example}. By \cite[Lemma 3.4]{shaked2016spn}, ${Q} \in {\cal COP}^n$. Since ${Q} \in {\cal COP}^n$ and ${c} \in \R^n_+$, it follows that (QP) is bounded below. Since $S \neq \emptyset$, we have $-\infty < \ell^* < +\infty$. Let
\[
{D} = \left[ \begin{matrix} \tilde{D} & 0 \\ 0 & 0 \end{matrix} \right] \in {\cal S}^n,
\]
where $\tilde{D}$ is defined as in Example~\ref{unb_example}. It follows that $D \in {\cal D}^n$, $\langle {A}^T {A}, {D} \rangle = 0$, and $\langle {Q}, {D} \rangle = \langle {\bar Q}, \tilde{D} \rangle < 0$. A similar argument as in Example~\ref{unb_example} reveals that $\ell_{\cal K}({\tilde x}) = -\infty$ for each ${\tilde x} \in S$, and therefore, $\ell_{\cal K} = -\infty$ by Proposition~\ref{props}.
\end{proof}

We remark that the proof of Proposition~\ref{gen_unb_dnn} yields an algorithmic recipe for constructing instances of (QP) with a finite optimal value such that the doubly nonnnegative relaxation is unbounded below. Furthermore, the construction in the proof of Proposition~\ref{gen_unb_dnn} can be extended to any ${\cal K} \in \mathbb{K}^*$, where ${\cal K} \neq {\cal CP}^{n+1}$, provided that there exists ${D} \in {\cal K} \backslash {\cal CP}^{n}$ such that $\langle {A}^T {A}, {D} \rangle = 0$, and $\langle {Q}, {D} \rangle < 0$.

Under Assumption~\ref{assum2}, we next establish a result similar to Proposition~\ref{conv_env_bou} for the copositive relaxation of instances of (QP) with an unbounded feasible region.

\begin{proposition} \label{conv_env_unb}
Let $S$ be an unbounded polyhedron and let ${\cal K} = {\cal CP}^{n+1}$. Then, under Assumption~\ref{assum2}, the convex underestimator $\ell_{\cal K}(\cdot)$ is the convex envelope of $q(\cdot)$ on $S$, i.e.,
\[
\ell_{\cal K}(x) = f_S(x), \quad x \in S,
\]
where $f_S(\cdot)$ denotes the convex envelope of $q(\cdot)$ on $S$ given by \eqref{def_conv_env}. Therefore, $\ell_{\cal K} = \ell^*$. In addition, $\ell_{\cal K}({\tilde x}) = \ell^*$ and the optimal solution is attained in $(P({\cal K},{\tilde x}))$ if and only if ${\tilde x} \in \textrm{conv}(S^*)$, where $S^*$ is given by \eqref{def_S_star}.
\end{proposition}
\begin{proof}
Our proof is similar to that of Proposition~\ref{conv_env_bou}, with slight modifications to account for the unboundedness of $S$. 

Let $\tilde{x} \in S$. Since $S$ is nonempty, the feasible region of $(P({\cal K},{\tilde x}))$ is nonempty by Corollary~\ref{emptyS}. Let $\tilde{Y} \in \cS^{n+1}$ be an arbitrary feasible solution of $(P({\cal K},{\tilde x}))$. By Lemma~\ref{burer_cp}, $\tilde{Y}$ admits a decomposition given by \eqref{burer_dec}, i.e., 
\[
{\tilde Y} = \left[ \begin{matrix} 1 & {\tilde x}^T \\ {\tilde x} & {\tilde X} \end{matrix} \right] = \sum\limits_{j=1}^{k_1} \lambda_j \left[ \begin{matrix} 1 \\ \tilde{x}^{j} \end{matrix} \right] \left[ \begin{matrix} 1 \\ \tilde{x}^{j} \end{matrix} \right]^T + \sum\limits_{j=1}^{k_2} \left[ \begin{matrix} 0 \\ \tilde{d}^{j} \end{matrix} \right] \left[ \begin{matrix} 0 \\ \tilde{d}^{j} \end{matrix} \right]^T,
\]
where $\tilde{x}^j \in S$,~$j = 1,\ldots,k_1$, $\lambda_j \geq 0,~j = 1,\ldots,k_1$ and $\tilde{d}^j \in L,~j = 1,\ldots,k_2$, where $L$ is given by \eqref{def_L}, such that $\sum\limits_{j=1}^{k_1} \lambda_j = 1$. Therefore,
\begin{eqnarray*}
\langle \widehat{Q}, \tilde{Y} \rangle & = & \sum\limits_{j = 1}^{k_1} \lambda_j q({\tilde x}^j) + \sum\limits_{j = 1}^{k_2} (\tilde{d}^j)^T Q \tilde{d}^j\\
 & \geq & \sum\limits_{j = 1}^{k_1} \lambda_j q({\tilde x}^j)\\
 & \geq & \sum\limits_{j = 1}^{k_1} \lambda_j f_S({\tilde x}^j)\\
 & \geq & f_S\left( \sum\limits_{j = 1}^{k_1} \lambda_j {\tilde x}^j \right) \\
 & = & f_S({\tilde x}),
 \end{eqnarray*}
where we used Assumption~\ref{assum2} in the second line, \eqref{def_conv_env} in the third line, the convexity of $f_S(\cdot)$ on the fourth line, and \eqref{burer_dec} in the last line. It follows that $\ell_{\cal K}({\tilde x}) \geq f_S({\tilde x})$, which, together with \eqref{def_conv_env} and Proposition~\ref{props}, implies that $\ell_{\cal K}(\tilde{x}) = f_S(\tilde{x})$. Therefore, $\ell_{\cal K}(\cdot)$ is the convex envelope of $q(\cdot)$ on $S$. 

The proof of the relation $\ell_{\cal K} = \ell^*$ is identical to that of Proposition~\ref{conv_env_bou}. 

As for the last assertion, if $\ell^* > -\infty$, then $S^* \neq \emptyset$ by \cite{FW56}. Therefore, the set of optimal solutions of $P({\cal K})$ is nonempty since $Y^* = [1;x^*][1;x^*]^T$ is an optimal solution of $P({\cal K})$ for any $x^* \in S^*$. Therefore, for each $\tilde{x} \in \textrm{conv}(S^*)$, a similar argument as in the proof of Proposition~\ref{conv_env_bou} reveals that $\ell_{\cal K}({\tilde x}) = \ell^*$ and the optimal solution of $(P({\cal K},{\tilde x}))$ is attained. The converse implication is proved similarly by using the decomposition \eqref{burer_dec} and Assumption~\ref{assum2}.

\end{proof}

We close this section by presenting the counterpart of Corollary~\ref{conv_env_bou_small_dim} for instances of (QP) with an unbounded feasible region. In contrast with Corollary~\ref{conv_env_bou_small_dim}, we need to impose an additional assumption for $n = 4$.

\begin{corollary} \label{conv_env_unb_small_dim}
Let $S$ be a nonempty polytope and let ${\cal K} \in \mathbb{K}^*$. For any $n \leq 3$, the convex underestimator $\ell_{\cal K}(\cdot)$ is the convex envelope of $q(\cdot)$ on $S$, i.e.,
\[
\ell_{\cal K}(x) = f_S(x), \quad x \in S,
\]
where $f_S(\cdot)$ denotes the convex envelope of $q(\cdot)$ on $S$ given by \eqref{def_conv_env}. Therefore, $\ell_{\cal K} = \ell^*$. In addition, $\ell_{\cal K}({\tilde x}) = \ell^*$  and the optimal solution is attained in $(P({\cal K},{\tilde x}))$ if and only if ${\tilde x} \in \textrm{conv}(S^*)$, where $S^*$ is given by \eqref{def_S_star}. Furthermore, all of the results above hold for $n = 4$ if there exists $\tilde{y} \in \R^m$ such that $A^T \tilde{y} = a \in \R^n_+$ and $b^T \tilde{y} = 1$.
\end{corollary}
\begin{proof}
The proof is essentially the same as the proof of Corollary~\ref{conv_env_bou_small_dim}, except that we now rely on Proposition~\ref{conv_env_unb} instead of Proposition~\ref{conv_env_bou}.
\end{proof}

We close this section by noting that the assumption for the $n = 4$ case is satisfied if, for instance, a subset of the variables in (QP) is bounded, as observed in \cite[Section 3]{Burer09}. 

\section{Concluding Remarks} \label{conc}

In this paper, we considered a family of so-called feasibility preserving convex relaxations of nonconvex quadratic programs. By observing that each such relaxation implicitly gives rise to a convex underestimator, we established several properties of such relaxations. In particular, our results imply Burer's well-known copositive representation result for nonconvex quadratic programs. 

We believe that our perspective in this paper is particularly useful since it sheds light onto how a convex relaxation obtained in higher dimensions translates back into the original dimension of the quadratic program. Furthermore, our results clearly highlight the important role played by the directions in the null space of the constraint matrix. An interesting research direction is to extend this perspective to a larger class of nonconvex optimization problems that can be formulated as copositive or generalized copositive optimization problems.

\bibliographystyle{abbrv}
\bibliography{references}
\end{document}